\documentclass[12pt]{article}
 
\usepackage[margin=1in]{geometry} 
\usepackage{amsmath,amsfonts,amsthm,dsfont,bbm,amssymb}
\usepackage{tikz}
\usepackage{hyperref}
\usetikzlibrary{matrix,arrows,decorations.pathmorphing,positioning}
\usepackage{tikz-cd}
\usepackage{authblk}

\theoremstyle{definition}
\newtheorem{definition}{Definition}
\theoremstyle{theorem}
\newtheorem{theorem}[definition]{Theorem}
\theoremstyle{proposition}
\newtheorem{proposition}[definition]{Proposition}
\theoremstyle{lemma}

\theoremstyle{definition}
\newtheorem{example}[definition]{Example}
\theoremstyle{corollary}
\newtheorem{corollary}[definition]{Corollary}
\theoremstyle{conjecture}

\begin{document}
 
 
 
\title{Commutative graded monads and localisable monads}
\author{\textbf{Rowan Poklewski-Koziell}}
\affil{University of Cape Town, South Africa}
\date{November 2021}
\maketitle
 
\section{Background}
 
Let us examine the following diagram of ``specialisations'' and ``generalisations'':

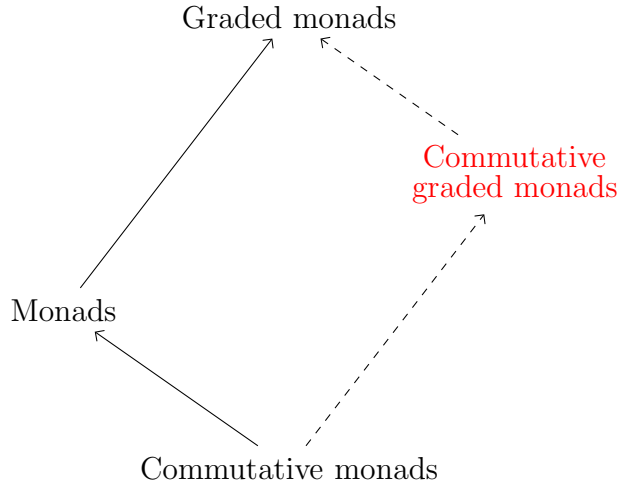
\begin{equation*}
\begin{tikzpicture}[scale = 1.5]

\node (U) at (0, 0) {Commutative monads};
\node (W) at (-2, 1.4) {Monads};
\node (V) at (2, 2.6) {
\shortstack{\textcolor{red}{Commutative} \\\textcolor{red}{graded monads}}
};
\node (Z) at (0, 4) {Graded monads};

\path[->,font=\scriptsize,>=angle 90]
(U) edge (W)
(U) edge[dashed] (V)
(W) edge (Z)
(V) edge[dashed] (Z);
\end{tikzpicture}
\end{equation*}

This is explained as follows: as originally observed by J. Ben\'abou \cite{Ben67}, a \textit{monad} in a bicategory $\mathcal{B}$ is nothing but a lax 2-functor $\mathbf{1} \longrightarrow \mathcal{B}$. This is the ``Monads'' node on the left. Specialising to ``Commutative monads'', these are exactly lax 2-functors $\mathbf{1} \longrightarrow \mathsf{MonCat}$, where $\mathsf{MonCat}$ is the $2$-category of monoidal categories, lax monoidal functors and monoidal natural transformations: they are \textit{monoidal} monads (see also, for example, \cite{Kock1972}).

We may generalise monads to ``Graded monads'' over arbitrary bicategories $\mathcal{B}$ as follows: for a monoidal category $\mathsf{M}$, an $\mathsf{M}$-\textit{graded monad over} $\mathcal{B}$ is a lax 2-functor $\Sigma(\mathsf{M}) \longrightarrow \mathcal{B}$, where $\Sigma(\mathsf{M})$ is the \textit{suspension} of $\mathsf{M}$. As we show below (and as has been remarked in \cite{OWE20}), in the specific case of $\mathcal{B} = \mathsf{Cat}$ such graded monads coincide with those described in \cite{FKM16}. In the latter paper, two notions of \textit{algebra} for a graded monad over $\mathsf{Cat}$ are introduced. While these algebras satisfy respective universal properties, it is not yet clear if they may be described as lax (or colax) limits of the lax 2-functors corresponding to the graded monad. Additionally, examples of the algebras of \cite{FKM16} have been presented as \textit{graded theories} in \cite[Section 3.3]{Kura20}. Graded monads over $\mathsf{Cat}$ have also been called \textit{parametric monads} \cite{Mel2012, Mel2012a}.

The purpose of the present work is the following: as has been shown in, for example, \cite{Day74}, that the that the category of Kleisli algebras for a monoidal monad carries a monoidal structure, we aim to specialise those graded monads over $\mathsf{Cat}$ whose category of ``Kleisli algebras'' carry a monoidal structure. This is the node on the right in our diagram above. We closely imitate and adapt several of the results and definitions of \cite{Zawa12}.

Let $\mathcal{D}$ be a 2-category with finite products of 0-cells. In such a 2-category $\mathcal{D}$, we can talk about \textit{monoidal objects}, \textit{(op)lax monoidal 1-cells}, and \textit{monoidal 2-cells}, as we talk about monoidal categories, (op)lax monoidal functors, and monoidal natural transformations in the 2-category $\mathsf{Cat}$. 

\begin{definition}{\label{monoidalObject}}
A \textit{monoidal object} in a 2-category $\mathcal{D}$ with finite products of 0-cells is a tuple $(C, \otimes_C, I_C, \alpha_C, \lambda_C, \rho_C)$ in which $C$ is a 0-cell
\begin{align*}
    \otimes_C: C \times C \longrightarrow C \qquad \text{and} \qquad I_C: 1_{\mathcal{D}} \longrightarrow C
\end{align*}
(where $1_{\mathcal{D}}$ is the terminal object of $\mathcal{D}$) are 1-cells, and
 \begin{align*}
        \alpha_C: \otimes_C \cdot (1_C \times \otimes_C) \longrightarrow &\otimes_C \cdot (\otimes_C \cdot 1_C) \qquad \qquad \lambda_C: \otimes_C \cdot \langle I_C, 1_C \rangle \longrightarrow 1_C \\ \\
        &\rho_C: \otimes_C \cdot \langle 1_C, I_C \rangle \longrightarrow 1_C
\end{align*}
are 2-cells which together make obvious pentagon and triangular diagrams commute. A \textit{lax monoidal morphism of monoidal objects} consists of a triple $\big(F, \phi, \overline{\phi}\big): (C, \otimes_C, I_C, \alpha_C, \lambda_C, \rho_C) \longrightarrow (C^\prime, \otimes_{C^\prime}, I_{C^\prime}, \alpha_C{^\prime}, \lambda_{C^\prime}, \rho_{C^\prime})$ in which:
\begin{itemize}
    \item $F: C \longrightarrow C^\prime$ is a 1-cell;
    \item $\phi: \otimes_{C^\prime} \cdot (F \times F) \longrightarrow F \cdot \otimes_C$ and $\overline{\phi}: I_{C^\prime} \longrightarrow F \cdot I_{C}$ are 2-cells
\end{itemize}
such that the following diagrams commute:

\begin{equation}\label{laxFunc1}
\begin{tikzpicture}[baseline=(current  bounding  box.center), scale = 1.5]

\node (A) at (0, 0) {$\otimes_{C^\prime} \cdot (F \times F ) \cdot (1_{C} \times \otimes_{C})$};
\node (B) at (0, 1.5) {$\otimes_{C^\prime} \cdot (1_{C^\prime} \times \otimes_{C^\prime}) \cdot (F \times F \times F)$};
\node (E) at (0, -1.5) {$F \cdot \otimes_{C} \cdot (1_{C} \times \otimes_{C} )$};
\node (C) at (5, 0) {$\otimes_{C^\prime} \cdot (F \times F) \cdot (\otimes_{C} \times 1_{C} )$};
\node (D) at (5, 1.5) {$\otimes_{C^\prime} \cdot (\otimes_{C^\prime} \times 1_{C^\prime} ) \cdot (F \times F \times F)$};
\node (F) at (5, -1.5) {$F \cdot \otimes_{C} \cdot (\otimes_{C} \times 1_{C} )$};

\path[->,font=\scriptsize,>=angle 90]
(B) edge node[left] {$\otimes_{C^\prime}(1_F \times \phi)$} (A)
(B) edge node[above] {$\alpha_{{C^\prime}{F, F, F}}$} (D)
(D) edge node[right] {$\otimes_{C^\prime}(\phi \times 1_F)$} (C)
(A) edge node[left] {$\phi(1_C \times \otimes_C)$} (E)
(C) edge node[right] {$\phi(\otimes_C \times 1_C)$} (F)
(E) edge node[below] {$F\alpha_{C}$} (F);
\end{tikzpicture}
\end{equation}

\begin{equation}\label{laxFunc2}
\begin{tikzpicture}[baseline=(current  bounding  box.center), scale = 1.5]

\node (A) at (0, 0) {$\otimes_{C^\prime} \cdot (F \times F) \cdot \langle 1_{C}, I_{C}\rangle$};
\node (B) at (0, 1.5) {$\otimes_{C^\prime} \cdot \langle 1_{C^\prime}, I_{C^\prime}\rangle \cdot F$};
\node (C) at (5, 0) {$F \cdot \otimes_C \cdot \langle 1_C, I_{C}\rangle$};
\node (D) at (5, 1.5) {$F$};

\path[->,font=\scriptsize,>=angle 90]
(B) edge node[left] {$\otimes_{C^\prime}(1_F \times \overline{\phi})$} (A)
(A) edge node[below] {$\phi \langle 1_C, I_C\rangle$} (C)
(B) edge node[above] {$\rho_{C^\prime}F$} (D)
(C) edge node[right] {$F\rho_{C}$} (D);
\end{tikzpicture}
\end{equation}

\begin{equation}\label{laxFunc3}
\begin{tikzpicture}[baseline=(current  bounding  box.center), scale = 1.5]

\node (A) at (0, 0) {$\otimes_{C^\prime} \cdot (F \times F) \cdot \langle I_{C}, 1_{C}\rangle$};
\node (B) at (0, 1.5) {$\otimes_{C^\prime} \cdot \langle I_{C^\prime}, 1_{C^\prime}\rangle \cdot F$};
\node (C) at (5, 0) {$F \cdot \otimes_C \cdot \langle I_{C}, 1_{C}\rangle$};
\node (D) at (5, 1.5) {$F$};

\path[->,font=\scriptsize,>=angle 90]
(B) edge node[left] {$\otimes_{C^\prime}(\overline{\phi} \times 1_F)$} (A)
(A) edge node[below] {$\phi \langle I_{C}, 1_{C}\rangle$} (C)
(B) edge node[above] {$\lambda_{C^\prime}F$} (D)
(C) edge node[right] {$F\lambda_{C}$} (D);
\end{tikzpicture}
\end{equation}

\noindent An \textit{oplax monoidal morphism of monoidal objects} consists of a triple $\big(F, \phi, \overline{\phi}\big): (C, \otimes_C, I_C, \alpha_C, \lambda_C, \rho_C) \longrightarrow (C^\prime, \otimes_{C^\prime}, I_{C^\prime}, \alpha_C{^\prime}, \lambda_{C^\prime}, \rho_{C^\prime})$ in which:
\begin{itemize}
    \item $F: C \longrightarrow C^\prime$ is a 1-cell;
    \item $\phi: F \cdot \otimes_C \longrightarrow \otimes_{C^\prime} \cdot (F \times F)$ and $\overline{\phi}:  F \cdot I_{C} \longrightarrow I_{C^\prime}$ are 2-cells
\end{itemize}
(note the opposite direction) making similar diagrams commute. A \textit{transformation of lax monoidal morphisms} $\tau: (F, \phi, \overline{\phi}) \longrightarrow (F^\prime, \phi^\prime, \overline{\phi^\prime})$ is a 2-cell $\tau: F \longrightarrow F^\prime$ such that the following diagrams commute:
\begin{equation}\label{transLaxMorph1}
\begin{tikzpicture}[baseline=(current  bounding  box.center), scale = 1.5]

\node (A) at (0, 0) {$F \cdot \otimes_C$};
\node (B) at (0, 1.5) {$\otimes_{C^\prime} \cdot (F \times F)$};
\node (C) at (4, 0) {$F^\prime \cdot \otimes_C$};
\node (D) at (4, 1.5) {$\otimes_{C^\prime} \cdot (F^\prime \times F^\prime)$};

\path[->,font=\scriptsize,>=angle 90]
(B) edge node[left] {$\phi$} (A)
(A) edge node[below] {$\tau \otimes_C$} (C)
(B) edge node[above] {$\otimes_{C^\prime}(\tau \times \tau)$} (D)
(D) edge node[right] {$\phi^\prime$} (C);
\end{tikzpicture}
\end{equation}

\begin{equation}\label{transLaxMorph2}
\begin{tikzpicture}[baseline=(current  bounding  box.center), scale = 1.5]

\node (A) at (0, 0) {$I_{C^\prime}$};
\node (B) at (-1.5, -1.5) {$F^\prime \cdot I_{C}$};
\node (C) at (1.5, -1.5) {$F \cdot I_{C}$};

\path[->,font=\scriptsize,>=angle 90]
(A) edge node[above left] {$\overline{\phi^\prime}$} (B)
(C) edge node[below] {$\tau I_C$} (B)
(A) edge node[above right] {$\overline{\phi}$} (C);
\end{tikzpicture}
\end{equation}
There are similar transformations of oplax morphisms.
\end{definition}

Denoting $2\mathsf{Cat}_\times$ the sub-3-category of $2\mathsf{Cat}$ full on 2-natural transformations and modifications whose 0-cells are 2-categories with finite products of 0-cells, and 1-cells are 2-functors preserving finite products of 0-cells, we have 3-functors:
\begin{align*}
    \mathsf{Mon}: 2\mathsf{Cat}_\times \longrightarrow 2\mathsf{Cat}_\times
\end{align*}
and
\begin{align*}
    \mathsf{Mon}_\text{op}: 2\mathsf{Cat}_\times \longrightarrow 2\mathsf{Cat}_\times
\end{align*}
given on objects by sending each 2-category $\mathcal{D}$ with finite products of 0-cells, to the 2-category $\mathsf{Mon}(\mathcal{D})$ (resp. $\mathsf{Mon}_\text{op}(\mathcal{D})$) of monoidal objects in $\mathcal{D}$, lax (resp. oplax) monoidal morphisms, and transformations of lax (resp. oplax) monoidal morphisms. There are also the obvious canonical forgetful 2-functors
\begin{align*}
    U_\mathcal{D}: \mathsf{Mon}(\mathcal{D}) \longrightarrow \mathcal{D}
\end{align*}
and 
\begin{align*}
    U_{\text{op}, \mathcal{D}}: \mathsf{Mon}_\text{op}(\mathcal{D}) \longrightarrow \mathcal{D}
\end{align*}
Furthermore, these 2-functors determine 3-natural transformations $U: \mathsf{Mon} \longrightarrow \mathsf{Id}_{2\mathsf{Cat}_\times}$ and $U_\text{op}: \mathsf{Mon}_\text{op} \longrightarrow \mathsf{Id}_{2\mathsf{Cat}_\times}$.

\begin{definition}
For a monoidal category $\mathsf{M}$, there is a bicategory $\Sigma(\mathsf{M})$ called the \textit{suspension} of $\mathsf{M}$. The bicategory $\Sigma(\mathsf{M})$ has one object $*$ and $\Sigma(\mathsf{M})(*, *) = \mathsf{M}$, and with composition given by the monoidal structure.
\end{definition}

Let us next show that, for a monoidal category $(\mathsf{M}, \otimes_M, I_M)$, an $\mathsf{M}$-graded monad (over $\mathsf{Cat}$) is nothing but a lax 2-functor $\Sigma(\mathsf{M}) \longrightarrow \mathsf{Cat}$. For, to give such a lax 2-functor (see, for example, \cite[Ch. 7.5]{Bor94}) is to give:
\begin{itemize}
    \item A category $\mathbf{C}$;
    \item A functor $F: \mathsf{M} \longrightarrow [\mathbf{C}, \mathbf{C}]$, whose value at some object $X \in \mathsf{M}$ we denote by $F_{X}$;
    \item A natural transformation $\gamma: \circ \cdot (F \times F) \longrightarrow F \cdot \otimes_M$, where, to distinguish compositions, $\circ$ is the composition of functors in $[\mathbf{C}, \mathbf{C}]$;
    \item A natural transformation $\delta: 1_{\mathbf{C}} \longrightarrow F_{I_M}$;
\end{itemize}
such that, for all objects $X, Y, Z$ in $\mathsf{M}$, the following diagrams commute:

\begin{equation}\label{first}
\begin{tikzpicture}[baseline=(current  bounding  box.center), scale = 1.5]

\node (D) at (0, 0) {$F_{X \otimes_M Y}F_Z$};
\node (B) at (0, 1.5) {$F_XF_YF_Z$};
\node (C) at (4, 0) {$F_{X \otimes_M Y \otimes_M Z}$};
\node (A) at (4, 1.5) {$F_XF_{Y \otimes_M Z}$};

\path[->,font=\scriptsize,>=angle 90]
(B) edge node[above] {$F_X\gamma_{Y, Z}$} (A)
(A) edge node[right] {$\gamma_{X, Y \otimes_M Z}$} (C)
(B) edge node[left] {$\gamma_{X, Y}F_Z$} (D)
(D) edge node[below] {$\gamma_{X \otimes_M Y, Z}$} (C);
\end{tikzpicture}
\end{equation}

\begin{equation}\label{second}
\begin{tikzpicture}[baseline=(current  bounding  box.center), scale = 1.5]

\node (D) at (0, 0) {$F_X$};
\node (B) at (0, 1.5) {$F_X \cdot 1_{\mathbf{C}}$};
\node (C) at (3, 0) {$F_{X \otimes_M I_M}$};
\node (A) at (3, 1.5) {$F_XF_{I_M}$};

\path[->,font=\scriptsize,>=angle 90]
(B) edge node[above] {$F_X\delta$} (A)
(A) edge node[right] {$\gamma_{X, I_M}$} (C)
(C) edge node[below] {$F_{\rho_X}$} (D);

\draw (B) edge[double equal sign distance] (D);
\end{tikzpicture}
\end{equation}

\begin{equation}\label{third}
\begin{tikzpicture}[baseline=(current  bounding  box.center), scale = 1.5]

\node (D) at (0, 0) {$F_X$};
\node (B) at (0, 1.5) {$1_{\mathbf{C}} \cdot F_X $};
\node (C) at (3, 0) {$F_{I_M \otimes_M X}$};
\node (A) at (3, 1.5) {$F_{I_M}F_X$};

\path[->,font=\scriptsize,>=angle 90]
(B) edge node[above] {$\delta F_X$} (A)
(A) edge node[right] {$\gamma_{I_M, X}$} (C)
(C) edge node[below] {$F_{\lambda_X}$} (D);

\draw (B) edge[double equal sign distance] (D);
\end{tikzpicture}
\end{equation}

\noindent A simple comparison with Definition \ref{monoidalObject} above easily shows then that this is nothing but a lax monoidal functor $F: \mathsf{M} \longrightarrow [\mathbf{C}, \mathbf{C}]$, where $[\mathbf{C}, \mathbf{C}]$ carries the obvious strict monoidal structure. Inspired by this example, we have the following:

\begin{definition}
For a monoidal category $(\mathsf{M}, \otimes_M, I_M)$ and bicategory $\mathcal{B}$, an $\mathsf{M}$-\textit{graded monad over} $\mathcal{B}$ is a lax 2-functor $\Sigma(\mathsf{M}) \longrightarrow \mathcal{B}$. To give such a lax 2-functor is equivalently to give a quadruple $(C, F, \gamma, \delta)$ in which:
\begin{itemize}
    \item $C$ is an object of $\mathcal{D}$;
    \item $F: \mathsf{M} \longrightarrow \text{End}_{\mathcal{D}}(C)$ is a functor, whose value at some object $X \in \mathsf{M}$ we denote by $F_{X}$;
    \item A natural transformation $\gamma: \circ \cdot (F \times F) \longrightarrow F \cdot \otimes_M$;
    \item A 2-cell $\delta: 1_{C} \longrightarrow F_{I_M}$;
\end{itemize}
such that, for all objects $X, Y, Z$ in $\mathsf{M}$, the similar diagrams of (\ref{first}), (\ref{second}) and (\ref{third}) commute. 
\end{definition}

We now give our main definition.

\begin{definition}\label{commGradMonad}
For a monoidal category $(\mathsf{M}, \otimes_M, I_M)$ and 2-category $\mathcal{D}$ with finite products of 0-cells, a \textit{commutative} $\mathsf{M}$-\textit{graded monad over} $\mathcal{D}$ is a lax 2-functor $ \Sigma(\mathsf{M}) \longrightarrow \mathsf{Mon}(\mathcal{D})$. Equivalently, this consists of an 11-tuple
\begin{align*}
    (C, \otimes_C, I_C, \alpha_C, \lambda_C, \rho_C, F, \Phi, \overline{\Phi}, \Gamma, \Delta)
\end{align*}
in which:
\begin{itemize}
    \item $(C, \otimes_C, I_C, \alpha_C, \lambda_C, \rho_C)$ (which we will abbreviate as $(C, \otimes_C, I_C)$ when the context is clear) is a monoidal object of $\mathcal{D}$;
    \item $(F, \Phi, \overline{\Phi})$ is a functor \begin{align*}
        \mathsf{M} \longrightarrow \text{End}_{\mathsf{Mon}(\mathcal{D})}\big((C, \otimes_C, I_C)\big)
    \end{align*}
    whose value at some object $X \in \mathsf{M}$ is the lax monoidal morphism denoted by $(F, \Phi, \overline{\Phi})(X) = (F_X, \Phi_X, \overline{\Phi}_X)$;
    \item $\Gamma: \circ \cdot (F, \Phi, \overline{\Phi}) \times (F, \Phi, \overline{\Phi}) \longrightarrow (F, \Phi, \overline{\Phi}) \cdot \otimes_M$ is a natural transformation;
    \item $\Delta: (1_{C}, 1, 1) \longrightarrow (F_{I_M}, \Phi_{I_M}, \overline{\Phi}_{I_M})$ is a transformation of lax monoidal morphisms
\end{itemize}
such that, for all objects $X, Y, Z$ in $\mathsf{M}$, the similar diagrams of (\ref{first}), (\ref{second}) and (\ref{third}) commute. 
\end{definition}

Unwinding this definition more, we have that, for each $X \in \mathsf{M}$, $F_X: C \longrightarrow C$ is a 1-cell, while $\Phi_X: F_X \cdot \otimes_C \longrightarrow \otimes_{C} \cdot (F_X \times F_X)$ and $\overline{\Phi}_X:  F_X \cdot I_{C} \longrightarrow I_{C}$ are 2-cells such that $(F_X, \Phi_X, \overline{\Phi}_X): (C, \otimes_C, I_C) \longrightarrow (C, \otimes_C, I_C)$ is a lax monoidal morphism. For each morphism $f: X \longrightarrow Y$ in $\mathsf{M}$, $(F, \Phi, \overline{\Phi})(f)$ -- which we denote by $Ff$ -- is a 2-cell (in $\mathcal{D}$) $F_X \longrightarrow F_Y$ which is additionally a transformation of lax monoidal morphisms $(F_X, \Phi_X, \overline{\Phi}_X) \longrightarrow (F_Y, \Phi_Y, \overline{\Phi}_Y)$. It is now clear that, for example, sending $X$ to $F_X$ and $f$ to $Ff$ determines a functor $\mathsf{M} \longrightarrow \text{End}_{\mathcal{D}}(C)$.

Furthermore, for each  pair $X$, $Y$ of objects of $\mathsf{M}$, $\Gamma_{X, Y}$ is a 2-cell (in $\mathcal{D}$) $F_XF_Y \longrightarrow F_{X \otimes_M Y}$ which additionally is a transformation of lax monoidal morphisms
\begin{align*}
    \big(F_XF_Y, F_X \Phi_Y \cdot \Phi_X(F_Y \times F_Y), F_X \overline{\Phi}_Y \cdot \overline{\Phi}_X \big) \longrightarrow \big(F_{X \otimes_M Y}, \Phi_{X \otimes_M Y}, \overline{\Phi}_{X \otimes_M Y}\big)
\end{align*}
Again, clearly, $\Gamma$ is a natural transformation
\begin{align*}
    \circ \cdot (F \times F) \longrightarrow F \cdot \otimes_M
\end{align*}
Finally, $\Delta: 1_C \longrightarrow F_{I_M}$ is a 2-cell which is additionally a transformation of lax monoidal morphisms $(1_{C}, 1, 1) \longrightarrow (F_{I_M}, \Phi_{I_M}, \overline{\Phi}_{I_M})$. Therefore we immediately conclude the obvious result:

\begin{proposition}
Each commutative $\mathsf{M}$-graded monad over a 2-category $\mathcal{D}$ is an $\mathsf{M}$-graded monad over $\mathcal{D}$.
\end{proposition}

\begin{example}
Take $\mathsf{M}$ to be the terminal category $\mathbbm{1}$, and $\mathcal{D}$ to be the 2-category $\mathsf{Cat}$. Applying this to Definition \ref{commGradMonad} gives a monoidal monad. That is, a commutative monad.
\end{example}

\begin{definition}\label{morphismGradedMonad}
Let $\mathsf{M}$ be a monoidal category and $\mathcal{D}$ a 2-category. Given $\mathsf{M}$-graded monads $(C, F, \gamma, \delta)$ and $(D, G, \mu, \eta)$ over $\mathcal{D}$, a \textit{lax morphism of} $\mathsf{M}$-\textit{graded monads} $(C, F, \gamma, \delta) \longrightarrow (D, G, \mu, \eta)$ is a lax natural transformation of the corresponding lax 2-functors. It consists of a pair $(\Omega, \omega)$ in which:
\begin{itemize}
    \item $\Omega: C \longrightarrow D$ is a 1-cell in $\mathcal{D}$;
    \item $\omega: \text{hom}_{\mathcal{D}}(C, \Omega) \cdot F \longrightarrow \text{hom}_{\mathcal{D}}(\Omega, C) \cdot G$ is a natural transformation -- that is, a family of 2-cells $\big(\omega_X : \Omega F_X \longrightarrow G_X \Omega\big)_{X \in \mathsf{M}}$
\end{itemize}
such that, for all objects $X, Y$ in $\mathsf{M}$, the following diagrams commute:

\begin{equation}\label{laxMorphGrad1}
\begin{tikzpicture}[baseline=(current  bounding  box.center), scale = 1.5]

\node (D) at (0, 0) {$\Omega F_{X \otimes_M Y}$};
\node (B) at (0, 1.5) {$\Omega F_XF_Y$};
\node (C) at (6, 0) {$G_{X \otimes_M Y}\Omega$};
\node (A) at (6, 1.5) {$G_XG_Y\Omega$};
\node (E) at (3, 1.5) {$G_X\Omega F_Y$};

\path[->,font=\scriptsize,>=angle 90]
(B) edge node[above] {$\omega_X F_Y$} (E)
(E) edge node[above] {$G_X \omega_Y$} (A)
(A) edge node[right] {$\mu_{X, Y}\Omega$} (C)
(B) edge node[left] {$\Omega\gamma_{X, Y}$} (D)
(D) edge node[below] {$\omega_{X \otimes_M Y}$} (C);
\end{tikzpicture}
\end{equation}

\begin{equation}\label{laxMorphGrad2}
\begin{tikzpicture}[baseline=(current  bounding  box.center), scale = 1.5]

\node (D) at (0, 0) {$\Omega \cdot 1_{C}$};
\node (B) at (0, 1.5) {$\Omega$};
\node (C) at (5, 0) {$\Omega F_{I_M}$};
\node (A) at (5, 1.5) {$G_{I_M}\Omega$};
\node (E) at (2.5, 1.5) {$1_{D} \cdot \Omega$};

\path[->,font=\scriptsize,>=angle 90]
(E) edge node[above] {$\eta \Omega$} (A)
(C) edge node[right] {$\omega_{I_M}$} (A)
(D) edge node[below] {$\Omega \delta$} (C);

\draw (B) edge[double equal sign distance] (D);
\draw (B) edge[double equal sign distance] (E);
\end{tikzpicture}
\end{equation}
\end{definition}

One obtains \textit{oplax morphisms of graded monads} as oplax natural transformations of the corresponding lax 2-functors. Note in particular, that an oplax morphism of graded monads $(C, F, \gamma, \delta) \longrightarrow (D, G, \mu, \eta)$ consists of a pair $(\Omega, \omega)$ in which:
\begin{itemize}
    \item $\Omega: C \longrightarrow D$ is a 1-cell in $\mathcal{D}$;
    \item $\omega$ is a family of 2-cells $\big(\omega_X :G_X \Omega  \longrightarrow \Omega F_X \big)_{X \in \mathsf{M}}$ (note the direction change!) natural in $X$
\end{itemize}
making similar diagrams commute.

\begin{definition}
For the $\mathsf{M}$-graded monads $(C, F, \gamma, \delta)$ and $(D, G, \mu, \eta)$ over a 2-category $\mathcal{D}$ in Definition \ref{morphismGradedMonad} above, and lax morphisms of $\mathsf{M}$-graded monads $(\Omega, \omega), (\Xi, \xi) : (C, F, \gamma, \delta) \longrightarrow (D, G, \mu, \eta)$, a \textit{transformation of lax morphisms of} $\mathsf{M}$-\textit{graded monads} $\beta: (\Omega, \omega) \longrightarrow (\Xi, \xi)$ is a modification of the corresponding lax natural transformations. That is, a 2-cell $\beta : \Omega \longrightarrow \Xi$ such that, for every pair $X, Y$ of objects, and every morphism $f: X \longrightarrow Y$ of $\mathsf{M}$, the following diagram commutes:

\begin{equation}\label{transOfLaxMor}
\begin{tikzpicture}[baseline=(current  bounding  box.center), scale = 1.5]

\node (A) at (0, 0) {$\Omega F_X$};
\node (B) at (0, 1.5) {$G_X \Omega$};
\node (C) at (3, 0) {$\Xi F_Y$};
\node (D) at (3, 1.5) {$G_Y \Xi$};

\path[->,font=\scriptsize,>=angle 90]
(B) edge node[left] {$\omega_X$} (A)
(A) edge node[below] {$\beta (Ff)$} (C)
(B) edge node[above] {$(Gf) \beta$} (D)
(D) edge node[right] {$\xi_Y$} (C);
\end{tikzpicture}
\end{equation}
\end{definition}

In a similar fashion, we obtain transformations of oplax morphisms. Once again, we have 3-functors:
\begin{align*}
    \Sigma_{\mathsf{M}}\mathsf{Mnd}: 2\mathsf{Cat} \longrightarrow 2\mathsf{Cat}
\end{align*}
and
\begin{align*}
    \Sigma_{\mathsf{M}}\mathsf{Mnd}_\text{op}: 2\mathsf{Cat} \longrightarrow 2\mathsf{Cat}
\end{align*}
given on objects by sending each 2-category $\mathcal{D}$, to the 2-category $\Sigma_{\mathsf{M}}\mathsf{Mnd}(\mathcal{D})$ (resp. $\Sigma_{\mathsf{M}}\mathsf{Mnd}_\text{op}(\mathcal{D})$) of $\mathsf{M}$-graded monads over $\mathcal{D}$, lax morphisms (resp. oplax) of graded monads, and transformations of lax (resp. oplax) morphisms of graded monads.

We now turn to algebras for $\mathsf{M}$-graded monads. Our definition is based on the fact that, given a (formal) monad $(C, t)$ in a 2-category $\mathcal{D}$, the \textit{Eilenberg-Moore object} for $(C, t)$ (if it exists) is nothing but the lax-limit of the lax 2-functor $\mathbf{1} \longrightarrow \mathcal{D}$ (as an object of the 2-category of lax 2-functors, lax natural transformations and modifications) corresponding to the monad, while the \textit{Kleisli object} for $(C, t)$ is the lax-colimit of the lax 2-functor (as an object of the 2-category of lax 2-functors, oplax natural transformations and oplax modifications) corresponding to the monad. Furthermore, in \cite[Theorem 13]{Street72}, Street shows that the two constructions of the paper (thought of as generalisations of the \textit{construction of Eilenberg-Moore algebras} and the \textit{construction of Kleisli algebras} \cite{Street72ftm} for a formal monad) capture the lax-limits (resp. lax-colimits) of so-called \textit{lax functors} (resp. \textit{oplax functors}) $\mathbf{A} \longrightarrow \mathsf{Cat}$, for a category $\mathbf{A}$ and $\mathsf{Cat}$ the 2-category of categories, functors and natural transformations. We, too, think of these notions as synonymous.

\begin{definition}
For an $\mathsf{M}$-graded monad $(C, F, \gamma, \delta)$ over a 2-category $\mathcal{D}$, the \textit{Eilenberg-Moore object} of the graded monad is the lax limit (if it exists) of the corresponding lax 2-functor as an object in $\Sigma_{\mathsf{M}}\mathsf{Mnd}(\mathcal{D})$. Similarly, the \textit{Kleisli object} of the graded monad is the lax colimit of the corresponding lax 2-functor as an object in $\Sigma_{\mathsf{M}}\mathsf{Mnd}_\text{op}(\mathcal{D})$.
\end{definition}

When all such lax limits (resp. lax colimits) exists in a particular 2-category, we have the following terminology.

\begin{definition}\label{EilenbergMoore}
For a monoidal category $\mathsf{M}$ and a 2-category $\mathcal{D}$, we say that $\mathcal{D}$ \textit{admits the construction of Eilenberg-Moore algebras} if the canonical inclusion 2-functor $\Delta: \mathcal{D} \longrightarrow \Sigma_{\mathsf{M}}\mathsf{Mnd}(\mathcal{D})$
has a right 2-adjoint
\begin{align*}
    \mathsf{EM} : \Sigma_{\mathsf{M}}\mathsf{Mnd}(\mathcal{D}) \longrightarrow  \mathcal{D} 
\end{align*}
Dually, $\mathcal{D}$ \textit{admits the construction of Kleisli algebras} if the inclusion 2-functor $\Delta_{\text{op}} : \mathcal{D} \longrightarrow \Sigma_{\mathsf{M}}\mathsf{Mnd}_\text{op}(\mathcal{D})$
has a left 2-adjoint
\begin{align*}
    \mathsf{K} : \Sigma_{\mathsf{M}}\mathsf{Mnd}_\text{op}(\mathcal{D}) \longrightarrow  \mathcal{D}
\end{align*}
\end{definition}

\section{Lifting EM and Kleisli objects under $\mathsf{Mon}$}

Note that, for a monoidal category $\mathsf{M}$, 2-categories $\mathcal{D}$ and $\mathcal{C}$ and a 2-functor $G: \mathcal{D} \longrightarrow \mathcal{C}$, the 2-functor \begin{align*}
    \Sigma_{\mathsf{M}}\mathsf{Mnd}_\text{op}(G): \Sigma_{\mathsf{M}}\mathsf{Mnd}_\text{op}(\mathcal{D}) \longrightarrow \Sigma_{\mathsf{M}}\mathsf{Mnd}_\text{op}(\mathcal{C})
\end{align*}
is given by sending an $\mathsf{M}$-graded monad $\mathcal{F}$ over $\mathcal{D}$ -- considered as a lax 2-functor $\Sigma(\mathsf{M}) \longrightarrow \mathcal{D}$ -- to the $\mathsf{M}$-graded monad $G \cdot \mathcal{F}$ over $\mathcal{C}$.
\begin{definition}
For a monoidal category $\mathsf{M}$, 2-categories $\mathcal{D}$ and $\mathcal{C}$ which admit the construction of Kleisli algebras, and a fixed 2-functor $G: \mathcal{D} \longrightarrow \mathcal{C}$, $G$ \textit{preserves Kleisli objects in its domain} when the diagram below
\begin{equation*}
\begin{tikzpicture}[baseline=(current  bounding  box.center), scale = 1.5]

\node (DT) at (0, 1.8) {$\Sigma_{\mathsf{M}}\mathsf{Mnd}_\text{op}(\mathcal{D})$};
\node (CS) at (3, 1.8) {$\mathcal{D}$};
\node (D) at (0, 0) {$\Sigma_{\mathsf{M}}\mathsf{Mnd}_\text{op}(\mathcal{C})$};
\node (C) at (3, 0) {$\mathcal{C}$};

\path[->,font=\scriptsize,>=angle 90]
(DT) edge[bend right] node[below] {$\mathsf{K}_{\mathcal{D}}$} (CS)
(CS) edge[bend right] node[above] {$\Delta_{\text{op}, \mathcal{D}}$} (DT)
(DT) edge node[left] {$\Sigma_{\mathsf{M}}\mathsf{Mnd}_\text{op}(G)$} (D)
(CS) edge node[right] {$G$} (C)
(D) edge[bend right] node[below] {$\mathsf{K}_{\mathcal{C}}$} (C)
(C) edge[bend right] node[above] {$\Delta_{\text{op}, \mathcal{C}}$} (D);
\end{tikzpicture}
\end{equation*}

\noindent commutes up to the canonical isomorphism. If this is the case when  $\mathcal{D}$ is the 2-category $\mathsf{Mon}(\mathcal{C})$, and $G = U_{\mathcal{C}}: \mathsf{Mon}(\mathcal{C}) \longrightarrow \mathcal{C}$ is the canonical forgetful 2-functor, we say that $\mathsf{Mon}(\mathcal{C})$ has \textit{standard Kleisli objects with respect to} $\mathcal{C}$. Dually, we may define \textit{standard Eilenberg-Moore objects with respect to} $\mathcal{C}$.
\end{definition}

\begin{proposition}\label{prop}
For a monoidal category $\mathsf{M}$ and 2-category $\mathcal{D}$, there is a canonical isomorphism of 2-categories
\begin{align*}
    \xi_\mathcal{D}: \Sigma_{\mathsf{M}}\mathsf{Mnd}_\text{op}\big(\mathsf{Mon}(\mathcal{D})\big) \longrightarrow \mathsf{Mon}\big(\Sigma_{\mathsf{M}}\mathsf{Mnd}_\text{op}(\mathcal{D})\big)
\end{align*}
which makes the following diagram commute
\begin{equation}\label{permutationIso}
\begin{tikzpicture}[baseline=(current  bounding  box.center), scale = 1.5]

\node (A) at (0, 0) {$\mathsf{Mon}\big(\Sigma_{\mathsf{M}}\mathsf{Mnd}_\text{op}(\mathcal{D})\big)$};
\node (B) at (0, 1.5) {$\mathsf{Mon}(\mathcal{D})$};
\node (C) at (4, 0) {$\Sigma_{\mathsf{M}}\mathsf{Mnd}_\text{op}(\mathcal{D})$};
\node (D) at (4, 1.5) {$\Sigma_{\mathsf{M}}\mathsf{Mnd}_\text{op}\big(\mathsf{Mon}(\mathcal{D})\big)$};

\path[->,font=\scriptsize,>=angle 90]
(B) edge node[left] {$\mathsf{Mon}(\Delta_{\text{op}, \mathcal{D}})$} (A)
(A) edge node[below] {$U_{\Sigma_{\mathsf{M}}\mathsf{Mnd}_\text{op}(\mathcal{D})}$} (C)
(B) edge node[above] {$\Delta_{\text{op}, \mathsf{Mon}(\mathcal{D})}$} (D)
(D) edge node[right] {$\Sigma_{\mathsf{M}}\mathsf{Mnd}_\text{op}(U_{\mathcal{D}})$} (C)
(D) edge node[above left] {$\xi_\mathcal{D}$} (A);
\end{tikzpicture}
\end{equation}
\end{proposition}

\begin{proof}
A 0-cell of $\Sigma_{\mathsf{M}}\mathsf{Mnd}_\text{op}\big(\mathsf{Mon}(\mathcal{D})\big)$ is precisely a commutative $\mathsf{M}$-graded monad over $\mathcal{D}$ as we have described in Definition \ref{commGradMonad}. On the other hand, a 0-cell of $\mathsf{Mon}\big(\Sigma_{\mathsf{M}}\mathsf{Mnd}_\text{op}(\mathcal{D})\big)$ is also an 11-tuple
\begin{align}\label{11tuple}
    \big((D, G, \mu, \eta), (\otimes_D, \phi), (I_D, \overline{\phi}), \alpha_D, \lambda_D, \rho_D\big)
\end{align}
in which:
\begin{itemize}
    \item $(D, G, \mu, \eta)$ is an $\mathsf{M}$-graded monad over $\mathcal{D}$;
    \item $(\otimes_D, \phi): (D \times D, (G, G), (\mu, \mu), (\eta, \eta)) \longrightarrow (D, G, \mu, \eta)$ is an oplax morphism of graded monads;
    \item $(I_D, \overline{\phi}): (1_{\mathcal{D}}, 1, 1, 1) \longrightarrow (D, G, \mu, \eta)$ is an oplax morphism of graded monads;
    \item $\alpha_D$, $\lambda_D$, $\rho_D$ are transformations of oplax morphisms of graded monads.
\end{itemize}
which altogether make (\ref{11tuple}) a monoidal object in $\Sigma_{\mathsf{M}}\mathsf{Mnd}_\text{op}(\mathcal{D})$. Now, routine but long checks show that 
\begin{align*}
    (D, \otimes_D, I_D, \alpha_D, \lambda_D, \rho_D)   
\end{align*}
is a monoidal object in $\mathcal{D}$ exactly because $(\otimes_D, \phi)$, $\alpha_D$, $\lambda_D$ and $\rho_D$ satisfy similar pentagonal and triangular diagrams. 

Furthermore, since $\alpha_D$, $\lambda_D$ and $\rho_D$ are transformations of oplax morphisms of graded monads, they satisfy oplax versions of the diagram (\ref{transOfLaxMor}), which is exactly to say that for each $X \in \mathsf{M}$, the triple
\begin{align*}
    (G_X, \phi_X , \overline{\phi}_X)     
\end{align*}
satisfies the oplax versions diagrams of (\ref{laxFunc1}), (\ref{laxFunc2}) and (\ref{laxFunc3}). That is, for each $X \in \mathsf{M}$, $(G_X, \phi_X , \overline{\phi}_X): (D, \otimes_D, I_D) \longrightarrow (D, \otimes_D, I_D)$ is an oplax monoidal morphism. Also, since $G: \mathsf{M} \longrightarrow \text{End}_{\mathcal{D}}(D)$ is a functor and because both of
\begin{align*}
    \phi:  \text{hom}_{\mathcal{D}}(\otimes_D, D) \cdot G \longrightarrow  \text{hom}_{\mathcal{D}}(D \times D, \otimes_D) \cdot (G, G)
\end{align*}
and
\begin{align*}
    \overline{\phi}: \text{hom}_{\mathcal{D}}(I_D, D) \cdot G \longrightarrow  \text{hom}_{\mathcal{D}}(1_{\mathcal{D}}, I_D)
\end{align*}
are natural transformations, for each morphism $f: X \longrightarrow Y$ of $\mathsf{M}$, $Gf$ satisfies the oplax versions of diagrams (\ref{transLaxMorph1}) and (\ref{transLaxMorph2}), and so $Gf: (G_X, \phi_X , \overline{\phi}_X) \longrightarrow (G_Y, \phi_Y , \overline{\phi}_Y)$ is a transformation of oplax morphisms. In total, this is to say that the assignments
\begin{align*}
    X \longmapsto (G_X, \phi_X, \overline{\phi}_X) \quad \text{ , } \quad f \longmapsto G_f
\end{align*}
define a functor $(G, \phi, \overline{\phi}) : \mathsf{M} \longrightarrow \text{End}_{\mathsf{Mon}(\mathcal{D})}\big((D, \otimes_D, I_D)\big)$.

Next to say that $\mu: \circ \cdot (G \times G) \longrightarrow G \cdot \otimes_M$ is a natural transformation, and that $(\otimes_D, \phi)$ and $(I_D, \overline{\phi})$ both satisfy oplax versions of (\ref{laxMorphGrad1}) is the same as to say that 
\begin{align*}
    \mu: \circ \cdot (G, \phi, \overline{\phi}) \times (G, \phi, \overline{\phi}) \longrightarrow (G, \phi, \overline{\phi}) \cdot \otimes_M
\end{align*}
is a natural transformation making a similar diagram to (\ref{first}) commute.

Finally, to say that $\eta: 1_{D} \longrightarrow G_{I_M}$ is a 2-cell such that $(\otimes_D, \phi)$ and $(I_D, \overline{\phi})$ both satisfy oplax versions of (\ref{laxMorphGrad2}) is the same as to say that 
\begin{align*}
    \eta: (1_{D}, 1, 1) \longrightarrow (G_{I_M}, \phi_{I_M}, \overline{\phi}{I_M})
\end{align*}
is a transformation oplax morphisms, making similar diagrams to (\ref{second}) and (\ref{third}) commute. 

Similar verifications may also be made for the associated 1- and 2-cells in each the above categories.
\end{proof}

\begin{theorem}
Let $\mathcal{D}$ be a 2-category with finite products of 0-cells that admits the construction of Kleisli algebras. Furthermore, suppose the left 2-adjoint $\mathsf{K}_{\mathcal{D}}: \Sigma_{\mathsf{M}}\mathsf{Mnd}_\text{op}(\mathcal{D}) \longrightarrow \mathcal{D}$ preserves finite products of 0-cells in its domain. Then, the 2-category $\mathsf{Mon}(\mathcal{D})$ admits the construction of Kleisli algebras and they are standard with respect to $\mathcal{D}$.
\end{theorem}

\begin{proof}
Since $\mathsf{Mon}$ is a 3-functor, it sends 2-adjunctions to 2-adjunctions. Therefore, in the diagram below, because $\mathsf{K}_{\mathcal{D}}$ preserves finite products of 0-cells we have that $\mathsf{Mon}(\mathsf{K}_\mathcal{D}) \dashv \mathsf{Mon}(\Delta_{\text{op}, \mathcal{D}})$.
\begin{equation*}
\begin{tikzpicture}[baseline=(current  bounding  box.center), scale = 1.5]

\node (B) at (0, 0) {$\Sigma_{\mathsf{M}}\mathsf{Mnd}_\text{op}(\mathcal{D})$};
\node (HIB) at (4, 0) {$\mathcal{D}$};
\node (BxHIA) at (0, 2) {$\mathsf{Mon}\big(\Sigma_{\mathsf{M}}\mathsf{Mnd}_\text{op}(\mathcal{D})\big)$};
\node (HIA) at (4, 2) {$\mathsf{Mon}(\mathcal{D})$};
\node (A) at (-1.9, 3.8) {$\Sigma_{\mathsf{M}}\mathsf{Mnd}_\text{op}\big(\mathsf{Mon}(\mathcal{D})\big)$};

\path[->,font=\scriptsize,>=angle 90]
(B.353) edge node[below] {$\mathsf{K}_{\mathcal{D}}$} (HIB.206)
(HIB.162) edge node[above] {$\Delta_{\text{op}, \mathcal{D}}$} (B.4)
(BxHIA) edge node[right] {$U_{\Sigma_{\mathsf{M}}\mathsf{Mnd}_\text{op}(\mathcal{D})}$} (B)
(HIA) edge node[right] {$U_{\mathcal{D}}$} (HIB)
(BxHIA.356) edge node[below] {$\mathsf{Mon}(\mathsf{K}_{\mathcal{D}})$} (HIA.187)
(HIA.169) edge node[above] {$\mathsf{Mon}(\Delta_{\text{op}, \mathcal{D}})$} (BxHIA.4)
(HIA) edge[bend right] node[above right] {$\Delta_{\text{op}, \mathsf{Mon}(\mathcal{D})}$} (A)
(A) edge[bend right] node[left] {$\Sigma_{\mathsf{M}}\mathsf{Mnd}_\text{op}(U_{\mathcal{D}})$} (B)
(A) edge node[above right] {$\xi_{\mathcal{D}}$} (BxHIA);
\end{tikzpicture}
\end{equation*}
Next, the diagram is commutative because diagram (\ref{permutationIso}) of Proposition \ref{prop} is commutative and because $U: \mathsf{Mon} \longrightarrow \mathsf{Id}$ is a 3-natural transformation. Finally, since $\xi_{\mathcal{D}}$ is an isomorphism of 2-categories, one has that
\begin{align*}
    \mathsf{Mon}(\mathsf{K}_\mathcal{D}) \cdot \xi_{\mathcal{D}} \dashv \Delta_{\text{op}, \mathsf{Mon}(\mathcal{D})}
\end{align*}
and so $\mathsf{Mon}(\mathcal{D})$ admits the construction of Kleisli algebras which are standard with respect to $\mathcal{D}$.
\end{proof}

Since the 2-functor $\mathsf{EM}_{\mathcal{D}}: \Sigma_{\mathsf{M}}\mathsf{Mnd}(\mathcal{D}) \longrightarrow \mathcal{D}$ is a right 2-adjoint, it automatically preserves finite products of 0-cells in its domain. Therefore, we have a slightly stronger result in the dual case.

\begin{theorem}\label{mainTheorem}
Let $\mathcal{D}$ be a 2-category with finite products of 0-cells that admits the construction of Eilenberg-Moore algebras. Then, the 2-category $\mathsf{Mon}_{\text{op}}(\mathcal{D})$ admits the construction of Eilenberg-Moore algebras and they are standard with respect to $\mathcal{D}$.
\end{theorem}

Since the 2-category $\mathsf{Cat}$ has all lax limits, we have the following.

\begin{corollary}
$\mathsf{Mon}_{\text{op}}(\mathsf{Cat})$ admits the construction of Eilenberg-Moore algebras and they are standard with respect to $\mathsf{Cat}$.
\end{corollary}

\section{Commutative graded monads and localisable monads}

In \cite[Lemma 27]{CDH21}, it is shown that for a stiff monoidal category $\mathbf{C}$, $\mathsf{ZI}(\mathbf{C})$-graded monads correspond bijectively to formal monads in the 2-category $[\mathsf{ZI}(\mathbf{C})^\text{op}, \mathsf{Cat}]$. In fact, this may be upgraded to an isomorphism of 2-categories.

\begin{proposition}
For a stiff monoidal category $\mathbf{C}$, there is an isomorphism of 2-categories between the 2-category of (formal) monads, morphism of monads, and monad morphism transformations in $[\mathsf{ZI}(\mathbf{C})^\text{op}, \mathsf{Cat}]$, and the 2-category $\Sigma_{\mathsf{ZI}(\mathbf{C})}\mathsf{Mnd}_{\text{op}}(\mathsf{Cat})$.
\end{proposition}

\begin{proof}

Suppose $(\overline{\mathbf{C}}, \overline{T}, \overline{\mu}^T, \overline{\eta}^T)$ and $(\overline{\mathbf{C}}, \overline{S}, \overline{\mu}^S, \overline{\eta}^S)$ are formal monads in $[\mathsf{ZI}(\mathbf{C})^\text{op}, \mathsf{Cat}]$. By \cite[Lemma 27]{CDH21}, these correspond bijectively to $\mathsf{ZI}(\mathbf{C})$-graded monads $(\mathbf{C}, T, \mu^T, \eta^T)$ and $(\mathbf{C}, S, \mu^S, \eta^S)$. For a morphism of monads
\begin{align*}
    (1_{\overline{\mathbf{C}}}, \overline{\phi}): (\overline{\mathbf{C}}, \overline{T}, \overline{\mu}^T, \overline{\eta}^T) \longrightarrow (\overline{\mathbf{C}}, \overline{S}, \overline{\mu}^S, \overline{\eta}^S)    
\end{align*}
we define an oplax morphism of $\mathsf{ZI}(\mathbf{C})$-graded monads $(1_\mathbf{C}, \phi): (\mathbf{C}, T, \mu^T, \eta^T) \longrightarrow (\mathbf{C}, S, \mu^S, \eta^S)$ as follows:

\begin{itemize}
    \item $1_\mathbf{C}$ is the usual identity functor;
    \item $\phi$ is the family of natural transformations $(\overline{\phi}_u: S_u \longrightarrow T_u)_{u \in \mathsf{ZI}(\mathbf{C})}$
\end{itemize}

\noindent The pair $(1_\mathbf{C}, \phi)$ indeed satisfies the oplax version of diagram (\ref{laxMorphGrad1}) because the following diagram commutes:

\begin{equation*}
\begin{tikzpicture}[baseline=(current  bounding  box.center), scale = 1.5]

\node (A) at (0, 0) {$FT_uT_v(A)$};
\node (B) at (3, 0) {$FT_{u\otimes 0}T_{0 \otimes v}(A)$};
\node (C) at (6, 0) {$FT^2_{u \otimes v}(A)$};
\node (D) at (0, 1) {$S_uT_v(A)$};
\node (E) at (3, 1) {$S_{u\otimes 0}T_{0 \otimes v}(A)$};
\node (F) at (6, 1) {$S_{u \otimes v}T_{u \otimes v}(A)$};
\node (H) at (0, 2) {$S_uS_vF(A)$};
\node (I) at (3, 2) {$S_{u \otimes 0}S_{0 \otimes v}F(A)$};
\node (J) at (6, 2) {$S^2_{u \otimes v}F(A)$};

\node (K) at (9, 2) {$S_{u \otimes v}F(A)$};
\node (L) at (9, 0) {$FT_{u \otimes v}(A)$};

\path[->,font=\scriptsize,>=angle 90]
(A) edge node[below] {$T_{\rho^{-1}}T_{\lambda^{-1}}(A)$} (B)
(B) edge node[below] {$T_{u\otimes !}T_{! \otimes v}(A)$} (C)

(D) edge node[below] {$S_{\rho^{-1}}T_{\lambda^{-1}}(A)$} (E)
(E) edge node[below] {$S_{u\otimes !}T_{! \otimes v}(A)$} (F)

(H) edge node[above] {$S_{\rho^{-1}}S_{\lambda^{-1}}(A)$} (I)
(I) edge node[above] {$S_{u\otimes !}S_{! \otimes v}(A)$} (J)

(D) edge node[left] {$\overline{\phi}_{u, T_{v}(A)}$} (A)

(E) edge node[right] {$\overline{\phi}_{u \otimes 0, T_{0 \otimes v}(A)}$} (B)

(H) edge node[left] {$S_{u}\overline{\phi}_{v, A}$} (D)
(I) edge node[left] {$S_{u \otimes 0}\overline{\phi}_{0 \otimes v, A}$} (E)
(J) edge node[right] {$S_{u \otimes v}\overline{\phi}_{u \otimes v, A}$} (F)
(F) edge node[right] {$\overline{\phi}_{u \otimes v, T_{u \otimes v}(A)}$} (C)
(J) edge node[above] {$\overline{\mu}^S_{u \otimes v, A}$} (K)
(C) edge node[below] {$\overline{\mu}^T_{u \otimes v, A}$} (L)
(K) edge node[right] {$\overline{\phi}_{u \otimes v, A}$} (L);
\end{tikzpicture}
\end{equation*}

A similar diagram will show the commutativity of the oplax version of diagram (\ref{laxMorphGrad2}).

On the other hand, given an oplax morphism of $\mathsf{ZI}(\mathbf{C})$-graded monads $(1_\mathbf{C}, \phi): (\mathbf{C}, T, \mu^T, \eta^T) \longrightarrow (\mathbf{C}, S, \mu^S, \eta^S)$, we get a morphism of monads $(\mathbf{C}, T, \mu^T, \eta^T) \longrightarrow (\mathbf{C}, S, \mu^S, \eta^S)$ because

\begin{equation*}
\begin{tikzpicture}[baseline=(current  bounding  box.center), scale = 1.5]

\node (A) at (0, 0) {$S_u(A)$};
\node (B) at (3, 2) {$S_uT_u(A)$};
\node (C) at (6, 0) {$T_u(A)$};
\node (D) at (0, 1) {$S_{u \otimes u}(A)$};
\node (F) at (6, 1) {$T_{u \otimes u}$};
\node (H) at (0, 2) {$S^2_u(A)$};
\node (J) at (6, 2) {$T^2_{u}(A)$};

\path[->,font=\scriptsize,>=angle 90]
(H) edge node[above] {$S_u\phi_{u, A}$} (B)
(D) edge node[below] {$\phi_{u \otimes u, A}$} (F)
(B) edge node[above] {$\phi_{u, T_{u}(A)}$} (J)
(A) edge node[below] {$\phi_{u, A}$} (C)
(H) edge node[left] {$\mu^S_{u, u, A}$} (D)
(D) edge node[left] {$S_{\nabla u, A}$} (A)

(J) edge node[right] {$\mu^T_{u, u, A}$} (F)
(F) edge node[right] {$T_{\nabla u, A}$} (C);
\end{tikzpicture}
\end{equation*}
commutes, where the top box commutes because $(1_\mathbf{C}, \phi)$ satisfies the oplax version of diagram (\ref{laxMorphGrad1}). 

Finally, it is trivial that transformations of such oplax morphisms of graded monads are exactly transformations of the corresponding morphisms of monads.
\end{proof}

Under the above isomorphism of 2-categories, one can see that the Eilenberg-Moore objects for formal monads in $[\mathsf{ZI}(\mathbf{C})^\text{op}, \mathsf{Cat}]$ correspond bijectively to the Eilenberg-Moore objects for the corresponding $\mathsf{ZI}(\mathbf{C})$-graded monads. For, denoting by $\Theta$ the above the above isomorphism of 2-categories, we have the following commutative diagram

\begin{equation*}
\begin{tikzpicture}[baseline=(current  bounding  box.center), scale = 1.5]

\node (A) at (0, -3.5) {$\mathsf{Cat}$};
\node (B) at (-2, -1.5) {$\Sigma_{\mathsf{ZI}(\mathbf{C})}\mathsf{Mnd}(\mathsf{Cat})$};
\node (C) at (2, -1.5) {$\mathsf{Mnd}\big([\mathsf{ZI}(\mathbf{C})^\text{op}, \mathsf{Cat}]\big)$};

\path[->,font=\scriptsize,>=angle 90]
(A) edge node[below left] {$\Delta_{\mathsf{Cat}}$} (B)
(C) edge node[above] {$\Theta$} (B)
(A) edge node[below right] {$\text{incl}$} (C);
\end{tikzpicture}
\end{equation*}
in which the inclusion 2-functor $\text{incl}: \mathsf{Cat} \longrightarrow \mathsf{Mnd}\big([\mathsf{ZI}(\mathbf{C})^\text{op}, \mathsf{Cat}]\big)$ is exactly the 2-functor of \cite[Theorem 13]{Street72}. Therefore, to give a right 2-adjoint to $\Delta_{\mathsf{Cat}}$ is to give a right 2-adjoint to $\text{incl}$. 
\section{Future work}

Our main interest in this work is in studying \textit{commutative} $\mathsf{ZI}(\mathbf{C})$-graded monads for a stiff monoidal category $\mathbf{C}$. Since we have the isomorphism $\Theta$ above, and a previously described isomorphism extending the bijective correspondence of \cite[Theorem 25]{CDH21}, we can ask, what are those localisable monads on a stiff monoidal category $\mathbf{C}$ which, as $\mathsf{ZI}(\mathbf{C})$-graded monads are in fact commutative $\mathsf{ZI}(\mathbf{C})$-graded monads? We may call such localisable monads \textit{commutative localisable monads}. 

For such commutative localisable monads, the Eilenberg-Moore objects for their corresponding commutative $\mathsf{ZI}(\mathbf{C})$-graded monads will have, by Theorem \ref{mainTheorem}, a monoidal structure inherited from $\mathbf{C}$. In turn, this should imply that the Eilenberg-Moore objects for their corresponding formal monads in $[\mathsf{ZI}(\mathbf{C})^\text{op}, \mathsf{Cat}]$ will have monoidal structures. This solves a previous open question posed in \cite{CDH21}.

We should be able to go further too. We should aim to establish a connection (perhaps an equivalence) between our construction of Eilenberg-Moore algebras in Definition \ref{EilenbergMoore} above and the Eilenberg-Moore objects of \cite[Section 3.1]{FKM16}. In particular, the latter have strong connections to graded theories in \cite[Theorem 20]{Kura20}. It is envisaged that we may be able to calculate free models for commutative graded monads corresponding to relatively simple commutative localisable monads. These models, having an additional commutative structure, may themselves serve as insightful models for certain aspects of concurrency.

\bibliographystyle{unsrt}
\bibliography{sample}

\begin{thebibliography}{10}

\bibitem{Ben67}
J.~B{\'e}nabou.
\newblock Introduction to bicategories.
\newblock In {\em Reports of the Midwest Category Seminar}, pages 1--77,
  Berlin, Heidelberg, 1967. Springer Berlin Heidelberg.

\bibitem{Kock1972}
A.~Kock.
\newblock Strong functors and monoidal monads.
\newblock {\em Archiv der Mathematik}, 23:113--120, 1972.

\bibitem{OWE20}
D.~Orchard, P.~Wadler, and H.~Eades.
\newblock Unifying graded and parameterised monads.
\newblock {\em Electronic Proceedings in Theoretical Computer Science},
  317:18–38, May 2020.

\bibitem{FKM16}
S.~Fujii, S.~Katsumata, and P.A Melli{\`e}s.
\newblock Towards a formal theory of graded monads.
\newblock In {\em Foundations of Software Science and Computation Structures},
  pages 513--530, Berlin, Heidelberg, 2016. Springer Berlin Heidelberg.

\bibitem{Kura20}
S.~Kura.
\newblock Graded algebraic theories.
\newblock In {\em Foundations of Software Science and Computation Structures},
  pages 401--421. Springer International Publishing, 2020.

\bibitem{Mel2012}
P.-A. Melli{\`e}s.
\newblock Parametric monads and enriched adjunctions.
\newblock {\em Syntax and Semantics of Low Level Languages}, 2012.

\bibitem{Mel2012a}
P.-A. Melli{\`e}s.
\newblock Game semantics in string diagrams.
\newblock In {\em Proceedings of Logic In Computer Science}, Dubrovnik, 2012.
  LICS.

\bibitem{Day74}
B.~Day.
\newblock On closed categories of functors {II}.
\newblock In {\em Category Seminar}, pages 20--54, 1974.

\bibitem{Zawa12}
M.~Zawadowski.
\newblock The formal theory of monoidal monads.
\newblock {\em Journal of Pure and Applied Algebra}, 216(8):1932--1942, 2012.

\bibitem{Bor94}
F.~Borceux.
\newblock {\em Handbook of Categorical Algebra}, volume~1 of {\em Encyclopedia
  of Mathematics and its Applications}.
\newblock Cambridge University Press, 1994.

\bibitem{Street72}
R.~Street.
\newblock Two constructions on lax functors.
\newblock {\em Cahiers de Topologie et Géométrie Différentielle
  Catégoriques}, 13(3):217--264, 1972.

\bibitem{Street72ftm}
R.~Street.
\newblock The formal theory of monads.
\newblock {\em Journal of Pure and Applied Algebra}, 2:149--168, 1972.

\bibitem{CDH21}
C.~Constantin, N.~Dicaire, and C.~Heunen.
\newblock Localisable monads.
\newblock 2021.

\end{thebibliography}

\end{document}